\documentclass{amsart}

\usepackage{color} 
\usepackage[automark]{scrpage2} 
\usepackage{mathrsfs,amsmath,marvosym,amsthm,amssymb} 
\usepackage[T1]{fontenc} 
\usepackage{paralist} 
\usepackage{nicefrac} 
\usepackage[arrow, matrix, curve]{xy} 
\usepackage{extpfeil} 

\usepackage{enumitem}
\usepackage[utf8]{inputenc}

\usepackage{url}

\usepackage{bbm}
\usepackage{setspace}

\usepackage{mathtools}

\usepackage{multibib}

\newcites{eigenes}{Eigene Publikationen}

\newcommand{\J}{\mathbf{J}}
\newcommand{\Ad}{\textup{Ad}}

\newcommand{\Rep}{\textup{Rep}}
\newcommand{\pidiff}{\pi_1^{\textup{diff}}}

\newcommand{\DE}{\textup{D.E.}}
\newcommand{\Soln}{\textup{Soln}}
\newcommand{\X}{\overline{X}}

\newcommand{\im}{\textup{im }}
\newcommand{\Q}{\mathbb{Q}}
\newcommand{\Z}{\mathbb{Z}}
\newcommand{\PP}{\mathbb{P}}

\newcommand{\F}{\mathscr{F}}
\newcommand{\IC}{\mathbb{C}}
\newcommand{\Gm}{\mathbb{G}_m}

\newcommand{\OO}{\mathcal{O}}

\newcommand{\LL}{\mathscr{L}}

\newcommand{\reg}{\textup{reg}}

\newcommand{\A}{\mathbb{A}}

\newcommand{\R}{\mathbb{R}}

\newcommand{\id}{\textup{id}}

\newcommand{\K}{\mathscr{K}}

\newcommand{\E}{\mathscr{E}}

\newcommand{\EE}{\mathscr{E}}

\newcommand{\END}{\mathscr{E}nd}

\newcommand{\M}{\mathcal{M}}
\newcommand{\El}{\textup{El}}
\newcommand{\Spec}{\textup{Spec}}
\newcommand{\Hom}{\textup{Hom}}
\newcommand{\End}{\textup{End}}

\newcommand{\KK}{\mathbb{C}(\!(t)\!)}
\newcommand{\KKu}{\mathbb{C}(\!(u)\!)}

\newcommand{\var}{\textup{var}}
\newcommand{\can}{\textup{can}}
\newcommand{\Hol}{\textup{Hol}}

\newcommand{\rk}{\textup{rk}}

\newcommand{\irr}{\textup{irr}}

\newcommand{\rig}{\textup{rig\,}}

\newcommand{\Weylz}{\mathbb{C}[z]\langle\partial_z\rangle}
\newcommand{\GL}{\textup{GL}}

\newcommand{\MC}{\textup{MC}}

\newcommand{\specialcell}[2][c]{%
  \begin{tabular}[#1]{@{}c@{}}#2\end{tabular}}

\newcommand{\Gdiff}{G_{\textrm{diff}}}
\newcommand{\Idiff}{I_{\textrm{diff}}}
\newcommand{\Gloc}{G_{\textrm{loc}}}

\theoremstyle{definition}
\newtheorem{defi}{Definition}[section]

\theoremstyle{plain}

\newtheorem{thm}[defi]{Theorem}
\newtheorem{lem}[defi]{Lemma}
\newtheorem{prop}[defi]{Proposition}
\theoremstyle{remark}
\newtheorem*{rem}{Remark}

\title{Classification of Rigid Irregular $G_2$-Connections}
\author{Konstantin Jakob}
\address{(K.J.) Massachusetts Institute of Technology, Department of Mathematics, Room 2-252B, 77 Massachusetts Ave, MA 02139}

\begin{document}
 
\maketitle

\begin{abstract}

Using the Katz-Arinkin algorithm we give a classification of
irreducible rigid irregular connections on a punctured $\PP^1_\IC$
having differential Galois group $G_2$, the exceptional simple
algebraic group, and slopes having numerator $1$. In addition to hypergeometric systems and their Kummer pull-backs we construct families of $G_2$-connections which are not of these types. 
	
\end{abstract}

\section{Introduction}

Rigid local systems are local systems which are determined up to isomorphism by the conjugacy classes of their local monodromies. Classically they arise as solution sheaves of certain regular singular differential equations, e.g. the Gaussian hypergeometric equation. 
In his book \cite{Ka96} Katz explains how to study rigid local systems using middle convolution. He proves that any irreducible rigid local systen can be obtained from a local system of rank one by iterating middle convolution and twists with other local systems of rank one. This provides a tool for the construction of rigid local systems.

Using this machinery Dettweiler and Reiter classified rigid local systems with monodromy group of type $G_2$ in \cite{Dett10} where $G_2$ is the simple exceptional algebraic group. It can be thought of as a subgroup of $SO(7)$ stabilizing the Dickson alternating trilinear form. As a consequence they proved that there is a family of motives for motivated cycles with $G_2$ as motivic Galois group answering a question raised by Serre. Other applications of rigid local systems include realizations of certain finite groups as Galois groups over $\mathbb{Q}$ in the framework of the inverse Galois problem, see for example \cite{DR00}. \par
In \cite{Arinkin10} Arinkin provides a generalization of Katz' existence algorithm to rigid connections with irregular singularities. In this case, a connection is called rigid if it is determined up to isomorphism by the restrictions to formal punctured discs around the singularities. This reduces to the classical notion if all singularities are regular singular.
Let $\IC[z]\langle \partial_z\rangle$ be the Weyl-algebra in one variable and denote by 
\[F: \IC[\tau]\langle \partial_\tau\rangle\rightarrow \IC[z]\langle \partial_z\rangle \]
the map defined by $F(\tau)=-\partial_z$ and $F(\partial_\tau)=z$. The Fourier-Laplace transform $\F(M)$ of a holonomic left $\Weylz$-module $M$ is then defined to be its pullback along the map $F$, i.e. it has the same underlying $\IC$-vector space but $\IC[\tau]\langle \partial_\tau\rangle$ acts through the map $F$. \par
Using this additional operation, Arinkin proves that given an irreducible rigid system of rank greater than one, there is a sequence of twists, coordinate changes and Fourier transforms such that the resulting system has lower rank. Combining this with a result of Bloch and Esnault in \cite{BlochEsnault04} on the rigidity of the Fourier transform of a rigid connection yields the desired algorithm. \par
In this article, we provide an extension of the result of Dettweiler and Reiter to a class of irregular singular connections of type $G_2$. One of the most important concepts for us are the slopes of an irregular singular connection. These are rational numbers measuring the irregularity. In particular, a singularity is regular singular
if all the slopes at this singularity vanish. The slopes of a
singularity are obtained through the Newton polygon of a differential
operator. The main result of this article is a classification of all rigid irregular connections with slopes having numerator $1$ and with differential Galois group of type $G_2$. \par

There are two main reasons for assuming the shape of the slopes. Since twists with a rank one connection preserve rigidity, the slopes of rigid systems are a priori unbounded. Still, most known examples of rigid connections of type $G_2$ and of connections of similar type have slopes of the desired shape. This includes for example the Frenkel-Gross connection from \cite[Section 5]{Fr09} and generalized hypergeometric modules as studied in \cite[Chapter 3]{Ka90}. \par
The second reason is of a technical nature. In contrast to the regular singular case, the structure of a connection at an irregular singular point is much more complicated. At a regular singular point the behaviour of a connection is basically given by the monodromy matrix obtained by analytic continuation of solutions close to this point, i.e. the datum describing the singularity is essentially the conjugacy class of a complex invertible matrix. In contrast, at an irregular singular point one considers the restriction of a connection $\E$ to a formal neighborhood of the point. In this way, one obtains a differential module over $\KK$ (or a $\KK$-connection). By a classical result of Levelt-Turrittin, any $\KK$-connection decomposes into finitely many so called elementary modules 
\[\El(\rho,\varphi,R)=\rho_+(\EE^\varphi\otimes R)\]
where $\rho$ is a ramification of some degree $p$, $\EE^\varphi
=(\KK,d+d\varphi)$ for some $\varphi\in \KK$ and some regular
connection $R$. To completely describe the structure of a connection at an irregular point, we therefore have to keep track of much more data. The assumption on the slopes is a way to limit the complexity of this.
\par 
%
The classification in particular contains the construction of previously not know rigid irreducible connections of type $G_2$ which are neither hypergeometric nor a pull-back of these. One of the key points in the construction is to understand how the formal structure of rigid connections at irregular singularities behaves with Fourier transform. For this we rely heavily on the formal stationary phase formula of López (cf. \cite{Lo04}) and explicit computations of the local Fourier transform of elementary modules by Sabbah in \cite{Sa08}. This allows us to explicitly compute the Levelt-Turrittin decomposition after Fourier transform. \par
To state our main result we will use the following notation. We will write $\El(p,\alpha,A)$ for the elementary module $\rho_{p,+}(\EE^{\frac{\alpha}{u}}\otimes R)$ where $R$ is the connection on $\Spec \,\KKu$ with monodromy $A$ and $\rho_p(u)=u^p$. By $\lambda \J(n)$ we denote a Jordan block of length $n$ with eigenvalue $\lambda\in \IC$ and we will omit $\J(1)$. 

\begin{thm} \label{classif}Let $\alpha_1,\alpha_2,\lambda,x,y,z\in \IC^*$ such that $\lambda^2\neq 1, \alpha_1\neq \pm \alpha_2, z^4\neq 1$ and such that $x,y, xy$ and their inverses are pairwise different and let $\varepsilon$ be a primitive third root of unity. Every formal type occuring in the following list is exhibited by is exhibited by a unique (up to isomorphism) irreducible rigid connection of rank $7$ on $\Gm$ with differential Galois group $G_2$. 
\begin{center}
\begin{tabular}{ c c }
$0$ & $\infty$ \\
\hline \\
$(\J(3),\J(3),1)$ & \specialcell[c]{$\El( 2,\alpha_1,(\lambda,\lambda^{-1}))$ \\ $\oplus\, \El(2,2\alpha_1,1) \oplus (-1)$} \\ [15pt]
$(-\J(2),-\J(2),E_3) $ & \specialcell[c]{$\El( 2,\alpha_1,(\lambda,\lambda^{-1}))$ \\ $\oplus\, \El(2,2\alpha_1,1)\oplus (-1)$} \\ [15pt]
$(xE_2,x^{-1}E_2,E_3)$ & \specialcell[c]{$\El( 2,\alpha_1,(\lambda,\lambda^{-1}))$ \\ $\oplus\, \El(2,2\alpha_1,1)\oplus (-1)$} \\ [15pt] 
\hline \\
$(\J(3),\J(2), \J(2))$ & \specialcell[c]{$\El(2,\alpha_1,1) \oplus \El(2,\alpha_2,1)$ \\ $\oplus\,\El(2,\alpha_1+\alpha_2,1) \oplus (-1)$} \\ [15pt]
\hline \\
$(iE_2,-iE_2,-E_2,1)$ & \specialcell[c]{$\El(3,\alpha_1,1)$ \\ $\oplus\,\El(3,-\alpha_1,1)\oplus(1)$} \\ [15pt]
\hline \\
$\J(7)$ & $\El(6,\alpha_1, 1)\oplus(-1)$ \\ [10pt]
$(\varepsilon\J(3), \varepsilon^{-1}\J(3),1)$ & $\El(6,\alpha_1, 1)\oplus(-1)$ \\ [10pt]
$(z\J(2), z^{-1}\J(2), z^2, z^{-2},1)$ & $\El(6,\alpha_1, 1)\oplus(-1)$ \\ [10pt]
$(x\J(2),x^{-1}\J(2), \J(3))$ & $\El(6,\alpha_1, 1)\oplus(-1)$ \\ [10pt]
$(x,y,xy,(xy)^{-1},y^{-1},x^{-1}, 1)$ & $\El(6,\alpha_1, 1)\oplus(-1)$ \\ [10pt]
\end{tabular}
\end{center}
Conversely, the above list exhausts all possible formal types of irreducible rigid irregular $G_2$-connections on open subsets of $\PP^1$ with slopes having numerator $1$.
\end{thm}
This provides a classification of irreducible rigid
connections with differential Galois group $G_2$ with slopes of the desired shape, in particular providing the aforementioned non-hypergeometric examples of such systems. We will discuss which systems arise as pullbacks in the final section after proving the main theorem. \par
This article is organized as follows. In Section \ref{tannaka} we briefly review
the Tannakian formalism for connections on a curve and for connections
over $\KK$, providing tools to classify connections with prescribed
differential Galois group and to compute some invariants of such
connections. In particular we introduce the upper numbering filtration on the local differential Galois group which is used to study irregular $\KK$-connections. \par 
In Section \ref{rigloc} we recall the definition of rigidity and of the index of rigidity and recall the refined Levelt-Turrittin decomposition of $\KK$-connections. From this we give invariants classifying the formal
type of a connection and explain how to obtain restrictions on the formal type of a rigid connection. \par
In Section \ref{KA-alg} we recall the operations needed for the Katz-Arinkin algorithm and discuss the principle of stationary phase. \par
Section \ref{G2conn} is dedicated to the study of the local and global structure of rigid irregular irreducible connections with differential Galois group $G_2$. This provides a rough classification of these connections, in particular yielding the result that any such connection has at most two singularities and can therefore be seen as a connection on $\Gm$. Finally, we conclude the proof of Theorem \ref{classif} in Section \ref{proof}.

%

\textbf{Acknowledgements.} The author would like to thank Michael
Dettweiler for his support during the writing of this article and
Stefan Reiter for various fruitful conversations concerning systems of
type $G_2$. The author was supported by the SPP 1489 during the
writing of this article.


\section{Tannakian Formalism for Connections}\label{tannaka}
Let $X$ be a smooth connected complex curve and denote by $\DE(X)$ the category of connections on $X$ as in \cite[1.1.]{Ka87}. By a connection we mean a locally free $\OO_X$-module $\EE$ of finite rank equipped with a connection map
\[\nabla: \EE\rightarrow \EE\otimes \Omega^1_{X/\IC}. \] 
Let $\X$ be the smooth compactification of $X$ and for any $x\in\X-X$ let $t$ be a local coordinate at $x$. The completion of the local ring of $\X$ at $x$ can be identified non-canonically with $\KK$. We define $\Psi_x(\EE)=\KK\otimes \EE$ to be the restriction of $\EE$ to the formal punctured disk around $x$.\par
Any $\Psi_x(\EE)$ obtained in this way is a $\KK$-connection, by which we mean a finite dimensional $\KK$-vector space admitting an action of the differential operator ring $\KK\langle\partial_t\rangle$. Its dimension will be called the \textit{rank} of the connection. The category of $\KK$-connections is denoted by $\textup{D.E.}(\KK)$. 
%

By \cite[Prop. 2.9]{vdP03} any $\KK$-connection $E$ is isomorphic to a connection of the form
\[\KK\langle\partial_t\rangle / (L) \]
for some operator $L\in \KK\langle\partial_t\rangle$ where $(L)$ denotes the left-ideal generated by $L$. To $L$ we can associate its Newton polygon $N(L)$ and the \textit{slopes} of $E$ are given by the slopes of the boundary of $N(L)$. These are independent of the choice of $L$. We call a $\KK$-connection \textit{regular singular} if all its slopes are zero. Any $\KK$-connection $E$ can be decomposed as
\[E=\bigoplus_{y\in \Q_{\ge 0}} E(y) \]
where only finitely many $E(y)$ are non-zero and where $\rk(E(y))\cdot y\in \Z_{\ge 0}$. The non-zero $y$ are precisely the slopes of $E$. We define the \textit{irregularity} of $E$ to be 
\[\irr(E):=\sum\ y\cdot\rk(E(y)).\]
It is always a non-negative integer. \par
Let $\EE$ be a connection on a smooth connected curve $X$ with smooth compactification $\X$ as before. We say that $\EE$ is \textit{regular singular} if the formal type $\Psi_x(\EE)$ at every singularity $x\in\X-X$ is regular singular. \par
%
By \cite[Section 1.1.]{Ka87} the category $\DE(X)$ is a neutral Tannakian category. Therefore there is a pro-algebraic group $\pidiff(X,x)$ such that we have an equivalence of categories
\[\DE(X)\rightarrow \Rep_\IC(\pidiff(X,x)).\]
Given a connection $\EE$ denote by $\rho_\EE:\pidiff(X,x)\rightarrow \GL(\omega_x(\EE))$ the associated representation. The image of $\rho$ is isomorphic to the differential Galois group $\Gdiff(\EE)$ of $\EE$. \par

Let $G$ be a connected reductive group over $\IC$. We will call algebraic homomorphisms $\pidiff(X,x)\rightarrow G(\IC)$ $G$-connections on $X$. Given a connection $\EE$ we can also consider it as a $\Gdiff(\EE)$-connection through the factorization
\[\xymatrix{\pidiff(X,x) \ar[rr]\ar[dr]^{\rho_\EE} & & \GL_n(\IC) \\
	& \Gdiff(\EE)(\IC) \ar@{^{(}->}[ur] 
}.\]\par
In the local setting there are similar notions. Let $K=\KK$ and consider the category $\DE(K)$ of $K$-connections. By \cite[II. 2.4.]{Ka87} there is as before a pro-algebraic group $\Idiff$ and an equivalence 
\[DE(K) \rightarrow \Rep_\IC(\Idiff)\]
coming from Tannakian formalism. 
%
%

Again if $\rho_E$ is the representation associated to $E$ its image $\im\rho_E=\Gloc(E)$ can be identified with the differential Galois group of $E$ considered as a differential module over $K$. Under the equivalence of $\DE(K)$ and $\Rep_\IC(\Idiff)$, horizontal sections correspond to invariant vectors. Hence we will sometimes abuse notation and write $E^{\Idiff}$ instead of $\Soln(E)$. \par
In addition, by \cite[II. 2.5.]{Ka87} there is a decreasing filtration $\Idiff^{(y)}$ indexed by $y\in \R_{>0}$ (called \textit{upper numbering filtration}) on $\Idiff$ with the property that for any connection $E$ with slopes $<y$ the kernel of its associated representation $\rho_E:\Idiff\rightarrow \GL(\omega(E))$ contains $\Idiff^{(y)}$. \par

%
Let $X$ be a smooth proper complex connected curve, $\Sigma$ a finite set of closed points of $X$ and $U=X-\Sigma$. In this situation, by \cite[II. 2.7]{Ka87} we can consider $\Gloc(\Psi_x(\EE))$ as a closed subgroup of $\Gdiff(\EE)$. This will allow us to deduce information about the differential Galois group of a connection from its formal type at the singularities. \par

\section{Rigid Connections and Local Data}\label{rigloc}
Let $X=\PP^1$ , $U$ a non-empty open subset of $X$ and $\EE\in \DE(U)$. We call the collection of isomorphism classes
\[\{[\Psi_x(\EE)]\}_{x\in X-U} \]
the \textit{formal type} of $\EE$, cf. \cite[2.1.]{Arinkin10}. 
We call a connection $\EE$ \textit{rigid} if it is determined up to isomorphism by its formal type. \par 
Fortunately there is a way to describe the structure of $\KK$-connections in a very explicit way, allowing for a classification of formal types. We introduce the following notation. For any formal Laurent series $\varphi\in \KKu$, non-zero ramification $\rho\in u\IC[[u]]$ and regular $\KKu$-connection $R$ we define
\[\El(\rho,\varphi,R):=\rho_+(\EE^\varphi\otimes R) \]
where $\rho_+$ denotes the push-forward connection and $\EE^\varphi$ is the connection 
\[(\KKu,d+d\varphi),\]i.e. it has an exponential solution $e^{-\varphi}$. Denote by $p$ the order of the ramification of $\rho$, by $q$ the order of the pole of $\varphi$ and by $r$ the rank of $R$. The connection $\El(\rho,\varphi,R)$ has a single slope $q/p$, its rank is $pr$ and its irregularity is $qr$.
\begin{thm}[Refined Levelt-Turrittin decomposition, \cite{Sa08}, Section 3]\label{LTthm} 
Let $E$ be a $\KK$-connection. There is a finite subset  $\Phi\subset \KKu$ such that 
 \[E\cong \bigoplus_{\varphi\in\Phi} \El(\rho_\varphi, \varphi, R_\varphi)\]
where $\rho_\varphi\in u\KKu\setminus \{0\}$ and $R_\varphi$ is a regular $\KKu$-connection. Denote by $p(\varphi)$ the order of $\rho_\varphi$. The decomposition is called \textit{minimal} if no $\rho_1,\rho_2$ and $\varphi_1$ exist such that $\rho_\varphi=\rho_1\circ\rho_2$ and $\varphi=\varphi_1\circ\rho_2$ and if for $\varphi,\psi\in\Phi$ with $p(\varphi)=p(\psi)$ there is no $p$-th root of unity $\zeta$ such that $\varphi=\psi\circ \mu_\zeta$ where $\mu_\zeta$ denotes multiplication by $\zeta$. In this case the above decomposition is unique.
\end{thm}
Therefore, to specify a connection $E$ over $\KK$ it is enough to give the finite set $\Phi$, the ramification maps $\rho_\varphi$ for all $\varphi\in\Phi$ and the monodromy of the connection $R_\varphi$. The latter can be given as a matrix in Jordan canonical form and we will use the notation $\lambda \J(n)$ for a Jordan block of length $n$ with eigenvalue $\lambda\in \IC$. For a general monodromy matrix we will write
\[(\lambda_1\J(n_1),...,\lambda_k\J(n_k)).\]
%

There is a criterion to identify rigid irreducible connections due to Katz in the case of regular singularities with a generalization by Bloch and Esnault in the case of irregular singularities.
\begin{prop}[\cite{BlochEsnault04}, Thm. 4.7. \& 4.10.] Let $\EE$ be an irreducible connection on $j:U\hookrightarrow \PP^1$. Denote by $j_{!*}$ the middle extension functor, cf. \cite[Section 2.9]{Ka90}. The connection $\EE$ is rigid if and only if 
\[\chi(\PP^1,j_{!*}(\END(\EE)))=2\]
where $\chi$ denotes the Euler-de Rham characteristic. 
\end{prop} 
For this reason, we set $\rig(\EE)=\chi(\PP^1,j_{!*}(\END(\EE))$ and call it the \textit{index of rigidity}. Whenever $\rig(\EE)=2$ we say that $\EE$ is \textit{cohomologically rigid}. The index of rigidity can be computed using local information only.
\begin{prop}[\cite{Ka90}, Thm 2.9.9.] Let $\EE$ be an irreducible connection on the open subset $j:U\hookrightarrow \PP^1$ and let $\PP^1-U=\{x_1,...,x_r\}$. The index of rigidity of $\EE$ is given as
\[\rig(\EE)=(2-r)\rk(\EE)^2-\sum_{i=1}^r \irr_{x_i}(\END(\EE))+\sum_{i=1}^r \dim_\IC\Soln_{x_i}(\END(\EE))\]
where $\Soln_{x_i}(\END(\EE))$ is the space of horizontal sections of $\Psi_{x_i}(\END(\EE))=\KK\otimes\END(\EE)$.
\end{prop}
Recall that $\Soln_{x_i}(\END(\EE))$ can be regarded as the space of invariants of the $\Idiff$-representation associated to $\Psi_{x_i}(\END(\EE))$. In the following we will see how to compute all local invariants appearing in the above formula provided we know the Levelt-Turrittin decomposition of the formal types at all points. Let $E$ be a $\KK$-connection with minimal Levelt-Turrittin decomposition
\[E=\bigoplus_i \El(\rho_i,\varphi_i,R_i).\]
Its endomorphism connection is then given by 
\[\label{endtensor} E\otimes E^*=\bigoplus_{i,j}\Hom(\El(\rho_i,\varphi_i,R_i),\El(\rho_j,\varphi_j,R_j)).\]
As the irregularity of $E\otimes E^*=\End(E)$ is given as sum over the slopes, it can be computed by combining this decomposition with \cite[Prop. 3.8.]{Sa08}.
%
Note that $\dim\Soln(E)=\dim\Soln(E^\reg)$ as any connection which is purely irregular has no horizontal sections over $\KK$ (otherwise it would contain the trivial connection). If $E$ has minimal Levelt-Turrittin decomposition $E=\bigoplus_i \El(\rho_i,\varphi_i,R_i)$, Sabbah shows in \cite[3.13.]{Sa08} that 
\begin{align}\label{centdecomp}\End(E)^{\reg}=\bigoplus_i\rho_{i,+}\End(R_i).\end{align}
A regular $\KKu$-connection $R$ is completely determined by its nearby cycles $(\psi_uR,T)$ with monodromy $T$. Its push-forward along any $\rho\in u\IC[[u]]$ of degree $p$ corresponds to the pair $(\psi_uR\otimes\IC^p,\rho_+T)$ with $\rho_+T$ given by the Kronecker product $T^{1/p}\otimes P_p$. Here $T^{1/p}$ is a $p$-th root of $T$ and $P_p$ is the cyclic permutation matrix on $\IC^p$. This is the formal monodromy of the push-forward connection. Let $V_{\rho_{+}R}$ be the $\Idiff$-representation associated to $\rho_+R$. We have
\[\dim\Soln(\rho_+R)=\dim V_{\rho_{+}R}^{\Idiff}=\dim \ker(\rho_+T-\id)=\dim\ker(T-\id). \]
In particular 
\begin{align*}\label{centdim}\tag{Z}
 \dim\Soln(\rho_+\End(R))&=\dim \ker(\rho_+\Ad(T)-\id) \\
 &=\dim\ker(\Ad(T)-\id) \\
 &=\dim\textup{Z}(T) 
\end{align*}
where $\textup{Z}(T)$ is the centraliser of $T$. Combining this with Formula \ref{centdecomp} allows us to compute $\dim\Soln(E)$ for any connection $E$ provided we know its Levelt-Turrittin decomposition. In particular, the condition that a connection $\EE$ is rigid provides us with restrictions on the irregularity and the centraliser dimensions of the monodromies of regular connections appearing in the Levelt-Turrittin decomposition. 

\section{The Katz-Arinkin Algorithm for Rigid Connections} \label{KA-alg}
We recall the various operations involved in the Arinkin algorithm as defined in \cite{Arinkin10}. Let $D_z=\IC[z]\langle \partial_z\rangle$ be the Weyl-algebra in one variable and $M$ a finitely generated left $D_z$-module. 
%

The Fourier isomorphism is the map 
\begin{align*} F:D_\tau&\rightarrow D_z \\
			\tau&\mapsto \partial_z \\
			\partial_\tau&\mapsto -z.
\end{align*} 
From now on we will always denote the Fourier coordinate by $\tau$ in the global setting. We will also use a subscript to indicate the coordinate on $\A^1$. Let $M$ be a finitely generated $D_z$-module on $\A^1_z$. The \textit{Fourier transform} of $M$ is 
\[\F(M)=F^*(M).\] 
Denote by $F^\vee: D_z\rightarrow D_\tau$ the same map as above with the roles of $z$ and $\tau$ reversed and let $\F^\vee=(F^\vee)^*$. \par 


 The functor $\F$ defines an auto-equivalence 
\[\F:\Hol(\A^1_z)\rightarrow \Hol(\A^1_\tau)\] 
of the category of holonomic $D_z$-modules on $\A^1$. 
We have $\F^\vee\circ \F=\varepsilon^*$ where $\varepsilon$ is the automorphism of $D_z$ defined by $\varepsilon(z)=-z$ and $\varepsilon(\partial_z)=-\partial_z$. \par
Using the Fourier transform we define the middle convolution as follows. For any $\chi\in \IC^*$ let $\K_\chi$ be the connection on $\Gm$ associated to the character $\pi_1(\Gm,1)\rightarrow \IC^\times$ defined by $\gamma \mapsto \chi$ where $\gamma$ is a generator of the fundamental group. We call $\K_\chi$ a \textit{Kummer sheaf}. Explicitly, $\K_\chi$ can be given as the trivial line bundle $\OO_{\Gm}$ equipped with the connection $d+\alpha d/dz$ for any $\alpha\in \IC$ such that $\exp(-2\pi i\alpha)=\chi$.\par
Let $i:\Gm\hookrightarrow \A^1$ be the inclusion. The \textit{middle convolution} of a holonomic module $M$ with the Kummer sheaf $\K_\chi$ is defined as
\[\MC_\chi(M):= \F^{-1}(i_{!*}(\F(M)\otimes \K_{\chi^{-1}})) \]
where $\F^{-1}$ denotes the inverse Fourier transform and $i_{!*}$ is the minimal extension. Note that $\F(\K_{\chi})=\K_{\chi^{-1}}$. \par
Given a connection $\EE$ on an open subset $j:U\hookrightarrow \A^1$ we can apply the Fourier transform or the middle convolution to its minimal extension $j_{!*}\EE$. We end up with a holonomic module on $\A^1$ which we can restrict in both cases to the complement of its singularities. This restriction is again a connection on some open subset of $\A^1$ and we denote it by $\F(\EE)$ for the Fourier transform and $\MC_\chi(\EE)$ for middle convolution. Whenever $\EE$ is defined on an open subset $U\subset \PP^1$ we can shrink $U$ such that $\infty\notin U$ and apply the above construction. \par
The Katz-Arinkin algorithm is given in the following theorem. It was proven in the case of regular singularities by Katz in \cite{Ka96} and in the case of irregular singularities by Arinkin in \cite{Arinkin10} (and presented in a letter by Deligne to Katz). 
\begin{thm} Let $\EE$ be an irreducible connection on an open subset $U\subset \PP^1$ and consider the following operations.
	\begin{enumerate}[noitemsep,label=(\roman{*})]
		\item Twisting with a connection of rank one,
		\item change of coordinate by a Möbius transformation,
		\item Fourier transform and
		\item middle convolution. 
	\end{enumerate}
The connection $\EE$ is rigid if and only if it can be reduced to a regular singular connection of rank one using a finite sequence of the above operations. 	
\end{thm}
%

The behaviour of the formal type of a connection under Fourier transform is governed by local Fourier transforms (as defined by Bloch and Esnault in \cite[Section 3]{BlochEsnault04}) and the principle of stationary phase. Let $E$ be a $\KK$-connection. The \textit{local Fourier transform} of $E$ from zero to infinity is obtained in the following way. Due to \cite[Section 2.4.]{Ka87} there is an extension of $E$ to a connection $\M_E$ on $\Gm$ which has a regular singularity at infinity and whose formal type at zero is $E$. We define
\[\F^{(0,\infty)}(E):=\F(\M_E)\otimes_{\IC[\tau]}\IC(\!(\theta)\!) \]
where $\tau$ is the Fourier transform coordinate and $\theta=\tau^{-1}$. In a similar fashion define for $s\in \IC^*$ transforms 
\[\F^{(s,\infty)}(E)=\E^{s/\theta}\otimes\F^{(0,\infty)}(E)\]
where $\E^{s/\theta}$ denotes as before the rank one connection with solution $e^{s/\theta}$. Recall that there also is a transform $\F^{(\infty,\infty)}$ which is of no interest to us, as it only applies to connections of slope larger than one. For details on this transform we refer to \cite[Section 3.]{BlochEsnault04}. \par
There are also transforms $\F^{(\infty,s)}$ which are inverse to $\F^{(s,\infty)}$, see \cite[Section 1]{Sa08}. For the local Fourier transforms Sabbah computed explicitly how the elementary modules introduced in the first section behave. The most important tool for controlling the formal type under Fourier transform is the formal stationary phase formula of López.
\begin{thm}[\cite{Lo04}, Section 1]
Let $M$ be a holonomic $D$-module on $\A^1$ with finite singularities $\Sigma$. There is an isomorphism
\[\Psi_\infty(\F(M)) \cong \bigoplus_{s\in \Sigma\cup\{\infty\}}\F^{(s,\infty)}(M). \]
\end{thm}
Let $M$ be a holonomic $\IC[[t]]\langle\partial_t\rangle$-module and choose an extension $\M$ as before. The formal type at infinity of the Fourier transform of this module is the local Fourier transform $\F^{(0,\infty)}(M)$. By [Sabbah, 5.7.], the local Fourier transform $\F^{(0,\infty)}(M)$ of a regular holonomic $\IC[[t]]\langle\partial_t\rangle$-module $M$ is the connection associated to the space of vanishing cycles $(\phi_tM,T)$ where $T=\id+\can\circ\var$. 
\begin{thm}[\cite{Sa08}, Section 5]\label{explicitstatphase} Let $\El(\rho, \varphi, R)$ be any elementary $\KK$-module with irregular connection. Recall that 
\[ \El(\rho,\varphi,R)=\rho_+(\E^{\varphi}\otimes R) \]
and that $q=q(\varphi)$ is the order of the pole of $\varphi$ which is positive by assumption. Denote by $'$ the formal derivative and let $\widehat{\rho}=\frac{\rho'}{\varphi'}$, $\widehat{\varphi}=\varphi-\frac{\rho}{\rho'}\varphi'$, $L_q$ the regular singular rank one connection with monodromy $(-1)^q$ and $\widehat{R}=R\otimes L_q$. The local Fourier transform of the elementary module is then given by
\[\F^{(0,\infty)}\El(\rho,\varphi,R)=\El(\widehat{\rho},\widehat{\varphi},\widehat{R}).\]
\end{thm}
In particular, we also have explicit descriptions 
\begin{align*}
\F^{(s,\infty)}\El(\rho,\varphi,R)&\cong\El(\widehat{\rho},\widehat{\varphi}+s/(\theta\circ \widehat{\rho}),\widehat{R}) \\
\F^{(s,\infty)}(M)&\cong\El(\id,s/\theta,\F^{(0,\infty)}M) \\
\end{align*}
for $M$ a regular $\IC[[t]]\langle\partial_t\rangle$-module.\par
Under twists with regular connections of rank one, elementary modules behave in the following way. Denote by $(\lambda)$ the regular $\KK$-connection with monodromy $\lambda\in \IC^*$. The following Lemma follows directly from the projection formula.  
\begin{lem}\label{Eltwist}
Let $\lambda\in \IC^*$, $\rho(u)=u^r$ and $\El(\rho,\varphi,R)$ be an elementary module. We have
\[\El(\rho,\varphi,R)\otimes (\lambda)\cong \El(\rho,\varphi,R\otimes (\lambda^r)).\]
\end{lem}
This in turn allows us to compute the change of elementary modules under middle convolution which we compute in terms of Fourier transforms and twist. \par 
\section{On connections of type $G_2$}\label{G2conn}
In this section we will restrict ourselves to irreducible rigid connections $\EE$ on non-empty open subsets of $\PP^1$ of rank $7$ with differential Galois group $\Gdiff(\EE)=G_2$ (where we fix the embedding $G_2\subset SO(7)\subset\GL_7$) and all of whose slopes have numerator $1$. 
 As connections with regular singularities of this type have already been classified by Dettweiler and Reiter, we will from now on assume that every irreducible rigid $G_2$-connection has at least one irregular singularity. We give a first approximation to the classification theorem of Section \ref{classif}.  \par
We will use the following notations. By $\rho_p$ we always denote the ramification $\rho_p(u)=u^p$, $R_k$ is a regular $\KKu$-connection of rank $k$ and $\varphi_q$ is a rational function of pole order $q$ at zero. A regular connection $R$ on the formal disc $\Spec\, \KKu$ is determined by its monodromy which can be given as a single matrix in Jordan canonical form. Let $A$ be a complex $n\times n$-matrix and $R$ the connection with monodromy $A$. We sometimes write 
\[\El(\rho_p,\varphi_q,A) \]
for the elementary module $\rho_{p,+}(\EE^\varphi \otimes R)$. Recall that by $\lambda\J(n)$ we denote a Jordan block of length $n$ with eigenvalue $\lambda\in\IC^*$, in particular $\J(n)$ is a unipotent Jordan block of length $n$. Additionally, $E_n$ is the identity matrix of length $n$. We will write 
\[(\lambda_1\J(n_1),...,\lambda_k\J(n_k))\] 
for a complex matrix in Jordan canonical form with eigenvalues $\lambda_1,...,\lambda_k$ and we will omit $\J(1)$. \par
\subsection{Local Structure}\label{locstructure}
Recall that we assume that all slopes of the irreducible rigid $G_2$-connections have numerator $1$. Additionally, a strong condition on the formal types is given by the self-duality which they have to satisfy. 
\begin{lem} Let $\EE$ be an irreducible rigid $G_2$-connection. The regular part of the formal type at any singularity $x$ of $\EE$ is of dimension $1$, $3$ or $7$. 
\end{lem}
\begin{proof} Let $x$ be any singularity of $\EE$. Denote by $E$ the formal type of $\EE$ at $x$ and write $E=E^\reg\oplus E^\irr$. This corresponds to a representation $\rho=\rho^\reg\oplus\rho^\irr$ of the local differential Galois group $I$ at $x$. First note that this representation has to be self-dual. We will show that purely irregular $\KK$-connections of odd dimension are never self-dual. Let $E$ be such a connection and write 
\[E=\bigoplus \El(p_i,\varphi_i,R_i) \]
for its minimal Levelt-Turrittin decomposition in which all the $\varphi_i$ are not in $\IC[[t]]$. For the dimension of $E$ to be odd, at least on of the elementary connections has to be odd dimensional, write $\El(p,\varphi,R)$ for that one. It's dual cannot appear in the above decomposition, as the dimension would not be odd in that case. So it suffices to prove that $\El(p,\varphi,R)$ itself is not self-dual. By \cite[Remark 3.9.]{Sa08} the dual of $\El(p,\varphi,R)$ is $\El(p,-\varphi,R^*)$. Thus a necessary condition for self-duality is \[\varphi\circ\mu_{\zeta_p}\equiv -\varphi \mod \IC[[u]].\]
Write $\varphi(u)=\sum_{i\ge -k}a_iu^i$ for some $k\in \Z_{\ge 0}$. The above condition translates to
\[\sum_{i\ge -k}a_i\zeta_p^i+u^i\sum_{i\ge -k}+a_iu^i \in \IC[[u]]. \]
Since $\varphi$ is supposed to be not contained in $\IC[[u]]$ there is an index $j<0$ such that $a_{j}\neq 0$. In this case we find that $a_j\zeta_p^j+a_j=0$, i.e. $\zeta_p^j=-1$. This can only hold if $p$ is even and in this case the dimension of $\El(p,\varphi,R)$ could not be odd. Therefore the dimension of the regular part of $E$ has to be odd. \par

Denote as before by $I^{(x)}$ the upper numbering filtration on $I=\Idiff$ and let $n=\dim E^\reg$. The smallest possible non-zero slope of $E$ is $1/6$. Since $\ker(\rho^\reg)$ contains $I^{(1/6)}$ we find 
\[\rho|_{I^{(1/6)}}= \mathbbm{1}^n\oplus \rho^\irr |_{I^{(1/6)}} \]
where $\mathbbm{1}$ denotes the trivial representation of rank one.

In the case $n=5$, the image of $\rho$ therefore contains elements of the form $(E_5,A)$ where $A$ is a non-trivial $2\times2$-matrix. By Table 1 in \cite{Dett10} such elements do not occur in $G_2(\IC)$.
\end{proof}
The following proposition is a special case of Katz's Main D.E. Theorem \cite[2.8.1]{Ka90}.
\begin{prop} Let $\EE$ be an irreducible rigid connection on $U\subset \PP^1$ of rank $7$ with differential Galois group $G_2$. If at some point $x\in \PP^1-U$ the highest slope of $\EE$ is $a/b$ with $a>0$ and if it occurs with multiplicity $b$, then $b=6$. 
\end{prop}
We will later see that the rigid $G_2$-connections we consider necessarily have exactly two singularities which we can choose to be zero and infinity. By \cite[Thm 3.7.1]{Ka90}, any system satisfying the conditions of the above proposition will then necessarily be hypergeometric. \par
One of the main ingredients in the proof of Katz's Main D.E. Theorem is the use of representation theory through Tannakian formalism as presented in the previous section. Applying the above Proposition (and self-duality) yields the possibilities listed in Table \ref{tab:slopedim} for the non-zero slopes and the respective dimensions in the slope decomposition of any irregular formal type of a rigid $G_2$-connection as considered above. 
\begin{center}
\begin{table}[htbp]
\caption{Possible slope decompositions}
\begin{tabular}{c c}
slopes & dimensions 
\\
\hline \\
$1$ 			& $4$		
\\ [3pt]
$1$ 			& $6$		
\\ [3pt]
$\frac{1}{2}, 1$ 	& $2, 2$	
\\ [3pt]
$\frac{1}{2}, 1$ 	& $2, 4$	
\\ [3pt]
$\frac{1}{2}, 1$ 	& $4, 2$	
\\ [3pt]
$\frac{1}{2}$ 	& $4$
\\ [3pt]
$\frac{1}{2}$ 	& $6$		
\\ [3pt]
$\frac{1}{3}$ 	& $6$		
\\ [3pt]
$\frac{1}{4}, 1$ 	& $4, 2$	
\\ [3pt]
$\frac{1}{6}$ 	& $6$ 		
\\ [3pt]

\end{tabular}
\label{tab:slopedim}
\end{table}
\end{center}

For an elementary module $\El(u^p,\varphi,R)$ with $\varphi\in \KKu$ we would like to describe the possible $\varphi$ more concretely. We have the following Lemma.
\begin{lem}\label{exptorusspecial}
The pole order of any $\varphi \in \KKu$ appearing in the Levelt-Turrittin decomposition into elementary modules of the formal type of a rigid irreducible connection of type $G_2$ with slopes of numerator $1$ can only be $1$ or $2$. 
\end{lem}
\begin{proof}
Suppose $\El(u^p,\varphi,R)$ appears in the formal type of such a system. Because the slopes all have numerator $1$, we have the following possibilities for $p$ and $q$ apart from $q=1$.
\begin{center}
\begin{tabular}{c c c}
$q$ & $p$  \\ [3pt]
\hline  \\
$2$ & $2,4,6$ \\ [5pt]
$3$ & $3,6$ \\ [5pt]
$4$ & $4$ \\ [5pt]
$6$ & $6$   \\ [5pt]
\end{tabular}
\end{center}
Note that in the cases $(q,p)=(6,6)$, $(q,p)=(4,4)$ and $(q,p)=(2,6)$, the module $\El(u^p,\varphi,R)$ cannot be self-dual. Indeed that would mean that $\varphi(\zeta u)=-\varphi(u)$ for some $\zeta$ with $\zeta^p=1$. Write $v=u^{-1}$. If $a_q$ denotes the coefficient of $v^q$ then the above condition means that 
\[a_q(\zeta u)^q=-a_qu^q, \]
i.e. $\zeta^q=-1$. This is a contradiction in these cases. The formal type of a connection of type $G_2$ has to be self-dual and therefore in the case that $q$ is even, the dual of $\El(\rho,\varphi,R)$ also has to appear in the formal type. If $p=4$ or $p=6$ this contradicts the fact that the rank of the connection is $7$. We are therefore left with the following cases.
\begin{center}
\begin{tabular}{c c c}
$q$ & $p$  \\ [3pt]
\hline  \\
$2$ & $2, 4$ \\ [5pt]
$3$ & $3,6$ \\ [5pt]
\end{tabular}
\end{center}
We analyze these cases separately. Suppose first we're in the case that $q=3$ and $p=6$. Then $\El(u^6,\varphi,R)$ is at least six dimensional, so $\dim R=1$ and the module has to be self-dual already. The isomorphism class of $\El(u^p,\varphi,R)$ depends only on the class of $\varphi$ mod $\IC[[u]]$, hence we think of $\varphi$ as a polynomial in $v=1/u$ vanishing at $v=0$. We can then write
\[\varphi(v)=a_3v^3+a_2v^2+a_1v \]
and self-duality implies that there is a $6$-th root of unity $\zeta$ such that 
\[a_3\zeta^3v^3=-a_3v^3.\]
Because $q=3$, $a_3\neq 0$ and we get that $\zeta^3=-1$. We have $a_2\zeta^2v^2=-a_2v^2$ implying that $a_2=0$. Therefore $\varphi$ is of the form
\[\varphi(v)=a_3v^3+a_1v. \]
In order to rule out this case we will need the exponential torus of an elementary module. Consider the module $E=\El(\sigma_p,\psi,L)$. Because of \cite[Lemma 2.4.]{Sa08} the \textit{exponential torus} of $E$ is the subgroup $\mathcal{T}$ of $(\IC^*)^p=\{(t_1,...,t_p)\}$ defined by $\prod t_i^{\nu_i}=1,\nu_i\in\Z$ for any relation of the form 
\[\prod \exp(\psi\circ\mu_{\zeta_p^i})^{\nu_i}=1\]
satisfied by the $\psi\circ\mu_{\zeta_p^i}$, see for example \cite[Section 11.22.]{Zoladek06}. The exponential torus can be considered as a subgroup of the local differential Galois group of $E$, i.e. $\mathcal{T}\subset G_2$ is a necessary condition for $\Gloc(E)\subset G_2$. \par
We claim that the torus attached to $\El(\rho,\varphi,R)$ for $\varphi(v)=a_3v^3+a_1v$ is three-dimensional. As the rank of $G_2$ is $2$, this means that no elementary module of this form can appear in any formal type. \par
If $a_1=0$, by \cite[Rem. 2.8.]{Sa08} we have 
\[\El(u^6,a_3u^{-3},R)\cong \El(u^2,a_3u^{-1},(u^3)_*R) \]
hence actually $q=1$ in this case. We can therefore assume that $a_1\neq 0$. Let $\zeta_6$ be a primitive $6$-th root of unity. We have to compute all relations of the form
\[\sum_{i=0}^5 k_i(a_3\zeta_6^{-3i}u^{-3}+a_1\zeta_6^{-i}u^{-1})=0,  k_i\in \Z. \]
Equivalently, we find all relations 
\[0=\sum_{i=0}^5 k_i(a_1\zeta_6^{-i} u^2+a_3\zeta_6^{-3i})=\sum_{i=0}^5k_i(a_1\zeta_6^{-i} u^2+(-1)^ia_3).\] 
First note that 
\[(a_1\zeta_6^{-i} u^2+a_3\zeta_6^{-3i})+(a_1\zeta_6^{-(i+3)} u^2+a_3\zeta_6^{-3(i+3)})=0\]
for $i=0,1,2$. Therefore any element in the exponential torus is of the form 
\[(x,y,z,x^{-1},y^{-1},z^{-1}).\]
It therefore suffices to prove that there are no further relations between the first three summands. Suppose there is a relation 
\[0=k(a_1 u^2+a_3)+l(-a_1\zeta_6^2u^2-a_3)+m(-a_1\zeta_6u^2+a_3)\]
with $k,l,m\in \Z$.
We find that $k=l-m$ and as $a_1\neq 0$ we conclude 
\begin{eqnarray*}0&=&l-m-\zeta_6^2l-\zeta_6m=l-m+\zeta_6^2m-\zeta_6m-\zeta_6^2l-\zeta_6^2m \\
&=&(\zeta_6^2-\zeta_6)m+l-m-(l+m)\zeta_6^2 \\
&=&l-2m-(l+m)\zeta_6^2, 
\end{eqnarray*}
using that $\zeta_6^2-\zeta_6=-1$. Therefore $l=-m$ and $-3m=0$, i.e. $m=0$. Finally, the exponential torus is given as
\[\mathcal{T}=\{(x,y,z,x^{-1},y^{-1},z^{-1})\}\in (\IC^*)^6\]
which is three-dimensional. Therefore a module of the above shape cannot appear in the formal type. \par
The case $q=3$ and $p=3$ works similarly. 

\end{proof}

We see that only the case $p=2$ and $q=2$ needs to be considered. The possible combinations of elementary modules in this case are either \begin{equation}\label{spcase1}\tag{S1}\El(\rho_2,\varphi_2,R_1)\oplus \El(\rho_2,-\varphi_2,R_1^*)\oplus R_3 \end{equation}
where $\varphi_2$ has a pole of order $2$ or
\begin{equation}\label{spcase2}\tag{S2}\El(\rho_2,\varphi_2,R_1)\oplus \El(\rho_2,-\varphi_2,R_1^*)\oplus \El(\rho_2,\varphi_1,R_1')\oplus R''_1 \end{equation}
where $\varphi_1$ has a pole of order $1$. \par
We can compute the irregularity and the dimension of the solution space in these cases through the use of \cite[Prop. 3.8.]{Sa08} and Formula \ref{centdecomp}. Using the formula $\dim\Soln(\rho_+\End(R))=\dim\textup{Z}(T)$ from the end of Section \ref{tannaka} we find that in the first case the dimension of the local solution space is one of $\{5,7,11\}$ and using the formulae of Section \ref{rigloc} we find that the irregularity is $20$. In the second case we find that the dimension of the solution space is $4$ and the irregularity is $39$. Apart from these two special cases all elementary modules appearing are of the form 
\[\El(\rho_p,\frac{\alpha}{u}, R_k) \]
with $\alpha \in \IC$. In this setting we can compute the dimension of the local solution space and its irregularity in the same way as we did for the two cases above. The resulting possible combinations for the local invariants at irregular singularities are listed in Table \ref{tab:locinv}.

\begin{center}
\begin{table}[htbp]
\caption{Local invariants}
\begin{tabular}{c c c c} 
slopes & dimensions & $\dim\Soln(\END)$ & $\irr(\END)$
\\ 
\hline \\
$1$ 			& $4$		& $5,7,9,11,13,17$	& $32, 36$	\\ [3pt]
$1$ 			& $6$		& $7,9,11,13,15,19$	& $30, 38, 42$	\\ [3pt]
$\frac{1}{2}, 1$ 	& $2, 2$	& $7,9,11,13,15$	& $29$		\\ [3pt]
$\frac{1}{2}, 1$ 	& $2, 4$	& $4, 6,10$		& $37, 39$	\\ [3pt]
$\frac{1}{2}, 1$ 	& $4, 2$	& $5,7$			& $30, 32$	\\ [3pt]
$\frac{1}{2}$ 		& $4$		& $5,7,9,11,13$	 	& $16, 18$	\\ [3pt]
$\frac{1}{2}$ 		& $6$		& $4,6,10$		& $15, 19, 21$	\\ [3pt]
$\frac{1}{3}$ 		& $6$		& $3$			& $12, 14$	\\ [3pt]
$\frac{1}{4}, 1$ 	& $4, 2$	& $4$			& $27$		\\ [3pt]
$\frac{1}{6}$ 		& $6$ 		& $2$			& $7$		\\ [3pt]

\end{tabular}
\label{tab:locinv}
\end{table}
\end{center}

\par 
\subsection{Global Structure}
Recall that the connection $\EE$ is rigid if and only if $\rig(\EE)=2$ where
\[\rig(\EE)=\chi(\PP^1, j_{!*}(\END(\EE))) \]
is the index of rigidity. If we denote by $x_1,...,x_r$ the singularities of $\EE$, the index of rigidity is given by
\[\rig(\EE)=(2-r)49-\sum_{i=1}^r \irr_{x_i}(\END(\EE))+\sum_{i=1}^r \dim_\IC \Soln_{x_i}(\END(\EE)). \]

\begin{lem}\label{numbsings} Let $\EE$ be an irreducible rigid $G_2$-connection on $U\subset \PP^1$ with singularities $x_1,...,x_r$ of slopes having numerator $1$. Then $2\le r\le 4$.
\end{lem}
\begin{proof} By Table 1 in \cite{Dett10} and by the table above we find that in any case
\[\dim_\IC \Soln_{x_i}(\END(\EE)) \le 29.\] 
As $\EE$ is rigid, we have
\[2=(2-r)49-\sum_{i=1}^r \irr_{x_i}(\END(\EE))+\sum_{i=1}^r \dim_\IC \Soln_{x_i}(\END(\EE)). \]
Therefore we get
\[2+(r-2)49+\sum_{i=1}^r \irr_{x_i}(\END(\EE)) \le 29r\]
and as $\irr_{x_i}(\END(\EE))\ge 0$ we conclude $20r-96\le 0$. This cannot hold for $r\ge 5$. If $r=1$, the first equality above shows $\irr_{x_1}\ge 47$ which again cannot hold by the table above.
\end{proof}
Let $\EE$ be an irreducible rigid $G_2$-connection with singularities $x_1,...,x_r$ where due to the above Lemma $r\in \{2,3,4\}$. We define $R(\EE)$ to be the tuple
\[(s_1,...,s_r,z_1,...,z_r)\in \Z_{\ge 0}^{2r}\]
with $s_i=\irr_{x_i}(\END(\EE))$ and $z_i=\dim_\IC\Soln_{x_i}(\END(\EE))$. The necessary condition on $R(\EE)$ for $\EE$ to be rigid is 
\[2=(2-r)49-\sum_{i=1}^r s_i+\sum_{i=1}^r z_i.\]
This condition provides the following list of possible invariants in the cases $r=2$ and $r=3$. Additionally, one finds that no cases with $r=4$ appear.
\begin{center}
 \begin{tabular}{c}
  $r=3$ \\
  \hline \\
  $(0, 0, 16, 25, 29, 13)$ \\
  $(0, 0, 16, 29, 29, 9)$ \\
  $(0, 0, 18, 29, 29, 11)$ \\
\end{tabular}
\end{center}
\begin{center}
\begin{tabular}{c c c}
&$r=2$&			\\
\hline			\\ 
$(0, 7, 7, 2)$		& $(0, 18, 13, 7)$	& $(0, 30, 25, 7)$ 	\\
$(0, 14, 13, 3)$	& $(0, 19, 11, 10)$	& $(0, 32, 25, 9)$	\\
$(0, 15, 7, 10)$	& $(0, 19, 17, 4)$	& $(0, 32, 29, 5)$	\\
$(0, 15, 11, 6)$	& $(0, 21, 13, 10)$	& $(0, 36, 25, 13)$	\\
$(0, 15, 13, 4)$	& $(0, 21, 17, 6)$	& $(0, 36, 29, 9)$	\\
$(0, 16, 7, 11)$	& $(0, 21, 19, 4)$	& $(0, 37, 29, 10)$	\\
$(0, 16, 9, 9)$		& $(0, 27, 25, 4)$	& $(0, 38, 25, 15)$	\\
$(0, 16, 11, 7)$	& $(0, 30, 13, 19)$	& $(0, 38, 29, 11)$	\\
$(0, 16, 13, 5)$	& $(0, 30, 17, 15)$	& $(0, 42, 29, 15)$	\\
$(0, 18, 9, 11)$	& $(0, 30, 19, 13)$	&	\\

\end{tabular}

\end{center}

Note that the two special cases (\ref{spcase1}) and (\ref{spcase2}) with $q=2$ do not appear. We can therefore classify the appearing elementary modules $\El(\rho_p,\varphi,R)$ by their ramification degree $p$, the coefficient $\alpha$ of $\varphi=\frac{\alpha}{u}$ and the monodromy of $R$. Now we can actually deal with the case $r=3$ by a case-by-case analysis using the Katz-Arinkin algorithm. \par 
$\mathbf{(0, 0, 16, 25, 29, 13)}$. According to Table \ref{tab:locinv}, the formal type at the irregular singularity has a $4$-dimensional part of slope $1/2$ and a $3$-dimensional regular part. In this case, the only possibility for the formal type is 
\[\El(\rho_2,\alpha/u,\pm E_2)\oplus (\pm E_3). \]
Since $G_2\subset SO(7)$ this formal type has to have a trivial determinant. By \cite[Proposition 2.9]{Sa08}, the regular part has to be $(E_3)$. Assume there exists a connection $\EE$ on $\PP^1-\{0,1,\infty\}$ with the above formal type at $\infty$ and local monodromy $(-E_4,E_3)$ and $(\J(2),\J(2), E_3)$ at $0$ and $1$ respectively. The formal type at infinity of the Fourier transform of this connection will be of the form 
\[(-E_4)\oplus \EE^{\frac{1}{u}}\oplus\EE^{\frac{1}{u}},\]
hence the Fourier transform has rank $6$. 
The formal type at $0$ will be of the form
\[\El(\rho_1,\widehat{\alpha}/u, \pm E_2)\oplus \J(2)^3 \]
which has rank $8$. This is a contradiction in both cases and we can exclude this case. A similar analysis rules out the remaining two cases. \par

%

We can therefore focus on the case $r=2$. A more thorough analysis of the shape of the elementary modules in question (applying the various criteria used up until now) shows that actually there are cases in which the irregularity $s_2$ does not occur with the local solution dimension $z_2$. After ruling these out we're left with the following list of tuples $R(\EE)$. 
\begin{center}
	\begin{tabular}{c c c}
		&$r=2$&			\\
		\hline			\\ 
		$(0, 7, 7, 2)$		& $(0, 16, 13, 5)$	& $(0, 32, 25, 9)$ 	\\
		$(0, 14, 13, 3)$	& $(0, 18, 9, 11)$	& $(0, 32, 29, 5)$	\\
		$(0, 15, 7, 10)$	& $(0, 18, 13, 7)$	& $(0, 36, 25, 13)$	\\
		$(0, 15, 11, 6)$	& $(0, 19, 17, 4)$	& $(0, 36, 29, 9)$	\\
		$(0, 15, 13, 4)$	& $(0, 21, 19, 4)$	& $(0, 37, 29, 10)$	\\
		$(0, 16, 7, 11)$	& $(0, 27, 25, 4)$	& $(0, 38, 29, 11)$	\\
		$(0, 16, 9, 9)$		& $(0, 30, 13, 19)$	& 	\\
		$(0, 16, 11, 7)$	& $(0, 30, 25, 7)$	& 	\\
	\end{tabular}
\end{center}
We would like to rule out further cases by computing the formal monodromy of the irregular formal type. For its definition in the general setting we refer to \cite[Section 1]{Mitschi96}. We will describe how to compute the formal monodromy of an elementary connection $\El(\rho,\varphi,R)$ where $\rho$ has degree $p$ and $R$ is a regular connection. We can choose a connection $R^{1/p}$ such that $\rho^+R^{1/p}\cong R$ (this boils down to choosing a $p$-th root of the monodromy associated to $R$). Now
\[\El(\rho,\varphi,R)=\rho_+(\EE^\varphi\otimes \rho^+R^{1/p})\cong \rho_+\EE^\varphi \otimes R^{1/p} \]
by the projection formula. Therefore by \cite[Lemma 2.4.]{Sa08}, 5, the differential equation associated to this elementary module has a formal solution of the form 
\[Y(t)=x^L e^{Q(t)} \]
where $x=t^p$, $Q(t)=\textup{diag}(\varphi(t), \varphi(\zeta_p t),...,\varphi(\zeta_p^{p-1}t))$ for a primitive $p$-th root of unity $\zeta_p$ and $L\in \textup{Mat}_n(\IC)$. The formal monodromy $A$ is defined such that $YA$ is the solution obtained by formal counter-clockwise continuation of $Y$ around $0$, see \cite[Chapter 3]{vdP03}.  \par
In the special case that $\varphi(t)=\alpha/t$ and $R$ is of rank one and corresponds to the monodromy $\lambda$, the formal monodromy is given as follows. Let $\lambda^{1/p}$ be a $p$-th root of $\lambda$ and choose $\mu$ such that $exp(2\pi i \mu)=\lambda^{1/p}$. The formal solution from above takes the form
\[Y(t)=x^\mu e^{Q(t)} \]
and the action of the formal monodromy sends $Y(t)$ to $\lambda^{1/p}x^\mu e^{\tilde{Q}(t)}$ where \[\tilde{Q}(t)=\textup{diag}(\varphi(\zeta_pt), \varphi(\zeta_p^2 t),...,\varphi(\zeta_p^{p-1}t),\varphi(t)).\]
Therefore in addition to multiplication by $\lambda^{1/p}$ the formal monodromy permutes the basis of the solution space, i.e. $A=\lambda^{1/p}P_p$ where $P_p$ denotes as before the cyclic permutation matrix. We will compute one example to show how to apply this discussion. \par
$\mathbf{(0, 16, 9, 9).}$ The formal type at the irregular singularity has to be of the form
\[\El(\rho_2,\alpha,R)\oplus (\J(2), 1)\]
or of the form 
\[\El(\rho_2,\alpha,R)\oplus (-E_2, 1)\]
where the connection $R$ corresponds to either $E_2$ or $-E_2$. In the first case we find that by the above discussion the formal monodromy is of the form $(E_2,-E_2,\J(2),1)$ or of the form $(iE_2,-iE_2,\J(2),1)$ both of which do not lie in $G_2(\IC)$. In the second case suppose that there exists a connection $\EE$ on $\Gm$ with the above formal type at $\infty$. The possibilities for the monodromy at $0$ are $(-\J(3),\J(3),-1)$, $(i\J(2),-i\J(2),-E_2,1)$ or $(x, -1, , -x, 1, -x^{-1},-1,x^{-1})$ where $x^4\neq 1$. In all these cases we compute 
\[\rk(\F(\EE\otimes \LL))=5\]
where $\LL$ is the rank one system with monodromy $-1$ at $0$ and $\infty$. But the formal type at $0$ of $\F(\EE\otimes \LL)$ would be of rank $7$. Therefore this case cannot occur.

All cases apart from the ones in the following list can be excluded by a combination of all the criteria we've used so far. We obtain constraints on the formal type at $\infty$ and can apply the Katz-Arinkin algorithm to obtain contradictions. 
\begin{center}
	\begin{tabular}{c}
		$r=2$ \\
		\hline \\
		$(0, 7, 7, 2)$ \\
		$(0, 14, 13, 3)$ \\
		$(0, 19, 17, 4)$ \\
		$(0, 21, 19, 4)$ \\
	\end{tabular}
\end{center}
Note that it might not suffice to simply apply one operation and compute the rank. We give an example of a case in which the computations are more complicated. \par 

$\mathbf{(0,38,29,11)}$. The monodromy at $0$ is $(\J(2),\J(2),E_3)$ and the formal type at $\infty$ has to be of the form
\[\El(\rho_1,\alpha,\lambda E_2)\oplus\El(\rho_1,-\alpha,\lambda^{-1} E_2)\oplus\El(\rho_1,2\alpha,\mu)\oplus\El(\rho_1,-2\alpha,\mu^{-1})\oplus (1).\]
Suppose there exists an irreducible connection $\EE$ on $\Gm$ with this formal type. We will apply Fourier transforms, twists and middle convolution to the connection $\EE$ to arrive at a contradiction. \par
Recall that $\F$ denotes the Fourier transform of connections and that $\MC_\chi$ is the middle convolution with respect to the Kummer sheaf $\K_\chi$. Let $\alpha_1,...,\alpha_r\in \IC^*$ such that $\alpha_1\cdot ...\cdot \alpha_r=1$. We denote by $\LL_{(\alpha_1,...,\alpha_{r+1})}$ the rank one connection on $\PP^1-\{x_1,...,x_r\}$ with monodromy $\alpha_i$ at $x_i$. For ease of notation we will write $(\alpha_1,...,\alpha_r)\otimes -$ for the twist $\LL_{(\alpha_1,...,\alpha_{r+1})}\otimes -$.  \par
We compute the change of local data in the following scheme in which we write the operation used in the first column and the formal type at the singularities in the other columns. \par
The way the data changes is given by the explicit stationary phase formula \ref{explicitstatphase} and Lemma \ref{Eltwist}. The $i$-th line is the result of applying the operation in the $(i-1)$-th line to the system in the $(i-1)$-th line. Writing $-$ in a column of a singularity means that this point is not singular.  \par
\medskip
\scalebox{0.6}{
	\begin{tabular}{c | c c c c c c} 
		& $0$ & $\alpha$ & $-\alpha$ & $2\alpha$ & $-2\alpha$ & $\infty$ \\ [5pt]
		\hline \\
		$\F$	& $(\J(2),\J(2),E_3)$ & $-$ & $-$ & $-$ & $-$ & \specialcell[c]{$\El(u,\alpha, \lambda E_2)\oplus\El(u,-\alpha, \lambda^{-1} E_2)$\\ $\oplus\, \El(u,2\alpha,\mu)\oplus\El(u,-2\alpha,\mu^{-1})\oplus (1)$} \\ [5pt]
		$(1,\lambda^{-1},\lambda,1,\mu,\mu^{-1})\otimes -$ & $\J(2)$ &$\lambda E_2$& $\lambda^{-1} E_2$& $(\mu,1)$ & $(\mu^{-1},1)$ & $E_2$ \\ [5pt]
		$\MC_{\mu^{-1}}$ & $\J(2)$ &$-$& $-$& $(\mu,1)$ & $(1,\mu)$ & $\mu^{-1} E_2$ \\ [5pt]
		& $\mu^{-1}$ & $-$ & $-$ & $\J(2)$ & $\J(2)$ & $\mu$ \\ [5pt]
	\end{tabular}}
	\medskip
	\par
In the last row we obtain a contradiction as the rank of the system is $1$, but its monodromy at $2\alpha$ resp. $-2\alpha$ is $\J(2)$. \par 
In the next section we will construct irreducible rigid $G_2$-connections in the four cases that are left which leads to the proof of the classification theorem for irreducible rigid irregular $G_2$-connections with slopes with numerator $1$.

\section{Proof of Theorem \ref{classif}} \label{proof}

We give the construction for the different cases. When varying the monodromy at zero in the same case, the construction is essentially the same up to twists with rank one systems. We will use the following notations. Denote by $\EE_{1,j}$ for $j=1,2,3$ the first three families, by $\EE_2$ the fourth family, by $\EE_3$ the fifth family and by $\EE_{4,j}$ for $j=1,...,5$ the final five families. Let $\mathscr{G}$ denote any operation on connections. We write $\mathscr{G}^k, k\in \Z_{>0}$, for its $k$-fold iteration. \\
\textbf{Construction of $\EE_{1,j}$.} Consider the connection 
\[\LL_{1,1}:=\LL_{(\lambda^{-1},-\lambda,\lambda^{-1},-\lambda)}\]
on $\PP^1-\{0,\frac{1}{4}\alpha_1^2, \alpha_1^2,\infty\}$ and the Möbius transform $\phi:\PP^1\rightarrow \PP^1, z\mapsto \frac{1}{z}$. Recall that $\F$ denotes the Fourier transform of connections. Our claim is that 
\[\EE_{1,1}:=\F(\phi^*(\F((1,-\lambda^{-1},1,-\lambda)\otimes \MC_{-\lambda}(\LL_{1,1}))))\]
has the formal type $(\J(3),\J(3),1)$ at $0$ and 
\[\El(2,\alpha_1,(\lambda,\lambda^{-1})) \oplus \El(2,2\alpha_1,1)\oplus (-1)\]
at $\infty$. Similar to before we compute the change of local data under the operations above in the following scheme.
\par
\medskip
\scalebox{0.7}{
\begin{tabular}{c | c c c c} 
			& $0$ & $\frac{1}{4}\alpha_1^2$ & $\alpha_1^2$ & $\infty$ \\ [5pt]
\hline \\
$\MC_{-\lambda}$	& $\lambda^{-1}$ & $-\lambda$ & $\lambda^{-1}$ & $-\lambda$\\ [5pt]
$(1,-\lambda^{-1},1,-\lambda)\otimes -$ & $(-1,1)$ &$(\lambda^2,1)$& $(-1,1)$& $-\lambda^{-1}E_2$ \\ [5pt]
$\F$ & $(-1,1)$ & $(-\lambda,-\lambda^{-1})$ & $(-1,1)$ & $E_2$ \\ [5pt]
$\phi^*$ & $(\J(2),\J(2))$ & $-$ & $-$ & \specialcell[c]{$\El(u,\frac{\alpha_1^2}{4u}, (-\lambda,-\lambda^{-1}))$ \\ $\oplus\, \El(u,\frac{\alpha_1^2}{u},-1)\oplus (-1)$} \\ [5pt] 
$\F$ & \specialcell[c]{$\El(u,\frac{\alpha_1^2}{4u}, (-\lambda,-\lambda^{-1}))$ \\ $\oplus\, \El(u,\frac{\alpha_1^2}{u},-1)\oplus (-1)$} 
&$-$&$-$& $(\J(2),\J(2))$  \\ [5pt]
 & $(\J(3),\J(3),1)$ & $-$ & $-$ & \specialcell[c]{$\El(\frac{4}{\alpha_1^2}u^2,\frac{\alpha_1^2}{2u}, (\lambda,\lambda^{-1}))$ \\ $\oplus\, \El(\frac{1}{\alpha_1^2}u^2,\frac{2\alpha_1^2}{u},1)\oplus (-1)$}
\end{tabular}}

\medskip \noindent
By Proposition \cite[Corollary 2.7.]{Sa08}, the connection 
\[\El(\frac{4}{\alpha_1^2}u^2,\frac{\alpha_1^2}{2u}, (\lambda,\lambda^{-1}))\oplus \El(\frac{1}{\alpha_1^2}u^2,\frac{2\alpha_1^2}{u},1)\oplus (-1)\]
is isomorphic to 
\[\El( 2,\alpha_1,(\lambda,\lambda^{-1})) \oplus \El(2,2\alpha_1,1)\oplus (-1). \]
This proves existence of the first type of connection. The same type of calculation shows that the connection 
\[
\EE_{1,2}:=\F\left((-1,-1)\otimes \phi^*\left(\F\left((1,\lambda^{-1},1,\lambda)\otimes\MC_\lambda\left(\LL_{(\lambda^{-1},\lambda,\lambda^{-1},\lambda)}\right)\right)\right)\right)
\]
exhibits the second formal type and the connection
\[\EE_{1,3}:=\F\left((x,x^{-1})\otimes \phi^*\left(\F\left((1,-\lambda^{-1}x^{-1},1,-\lambda x)\otimes \MC_{-\lambda x^{-1}}\left(\LL_{(\lambda^{-1},-\lambda x,\lambda^{-1},-\lambda x^{-1})}\right)\right)\right)\right)
\]
exhibits the third formal type. 
\\
\textbf{Construction of $\EE_2$.} For the second formal type at infinity, consider the connection $\LL_2:=\LL_{(-1,-1,-1,-1,1)}$ on $\PP^1-\{0,\frac{1}{4}\alpha_1^2,\frac{1}{4}\alpha_2^2,\frac{1}{4}(\alpha_1+\alpha_2)^2,\infty \}$. The connection 
\[\EE_2:=\F(\phi^*\F(\LL_2)))\]
has the desired formal type $(\J(3),\J(2), \J(2))$ at $0$ and 
\[\El(2,\alpha_1,1) \oplus \El(2,\alpha_2,1)\oplus\El(2,\alpha_1+\alpha_2,1) \oplus (-1)\]
at $\infty$. The computation works the same way as before. \\
\textbf{Construction of $\EE_3$.} For the third type consider the connection 
\[\LL_3:= \LL_{(-i,-\lambda,-\lambda^{-1},i)}\]
on $\PP^1-\{0,\frac{1}{27}\alpha_1^3,-\frac{1}{27}\alpha_1^3,\infty \}$. The system
\[\EE_4:=\F(\phi^*((-1,-1)\otimes \F((i,-i)\otimes \phi^*(\F((i,1,1,-i)\otimes \MC_i(\LL_3)))))) \]
has the required formal type. \\
\textbf{Construction of $\EE_{4,j}$.} For the final type we consider $\PP^1-\{0,\frac{1}{6^6}\alpha_1^6,\infty \}$. The formal types are then exhibited (in the order that they appear in the list) by the connections
\begin{align*}\EE_{4,1}&=\F\left((\phi^*\circ \F)^5\left(\LL_{(-1,-1,1)}\right)\right) \\
\EE_{4,2}&=\F\left((\varepsilon,\varepsilon^{-1})\otimes (\phi^*\circ\F)^3\left((\varepsilon^{-2},\varepsilon^2)\otimes(\phi^*\circ\F)^2\left(\LL_{(-\varepsilon,-\varepsilon^2,1)}\right)\right)\right), \\
\EE_{4,3}&=\F((z^{-2},z^2)\otimes (\phi^*\circ\F(
           (z^4,z^{-4})\otimes (\phi^*\circ\F) \\
&( (z,z^{-1})\otimes
           (\phi^*\circ \F)^2( (z^2,z^{-2})\otimes (\phi^*\circ \F)(\LL_{(-z^{-1},-z,1)}))))\\
\EE_{4,4}&=\F\left((\phi^*\circ \F)^2\left((x,x^{-1})\otimes (\phi^*\circ\F)^2\left((x^{-2},x^2)\otimes(\phi^*\circ\F)\left(\LL_{(-x,-x^{-1},1)}\right)\right)\right)\right), \\
\EE_{4,5}&=\F((x,x^{-1})\otimes(\phi^*\circ\F)((x^{-2},x^2)\otimes(\phi^*\circ\F)\\ 
&((xy^{-1},x^{-1}y)\otimes (\phi^*\circ\F)((y^{-2},y^2)\otimes (\phi^*\circ\F) \\
&((x,x^{-1})\otimes (\phi^*\circ\F)(\LL_{(-(xy)^{-1},-(xy)^{-1},x^2y^2})))))). \end{align*}
\textbf{The differential Galois groups.}
We compute the differential Galois group $G$ of the above types using an argument of Katz from \cite[§4.1.]{Ka90}. Let $\EE_1:=\EE_{1,1}$ and $\EE_4:=\EE_{4,1}$. The following proof works the same for all $\EE_{1,j}, j=1,2,3$. Note that all formal types are self-dual. Thus for $i=1,...,4$ we have that 
\[\Psi_x(\EE_i)\cong \Psi_x(\EE_i^*)\]
for $x=0,\infty$ and by rigidity we get $\EE_i\cong \EE_i^*$, i.e. all the above systems are globally self-dual. In addition the determinants are trivial meaning that actually $G\subset SO(7)$. We will focus first on the cases $i=1,2,3$. Let $G^0$ denote the identity component of $G$. By the proof of \cite[25.2]{Ka12} there are now only three possibilities for $G^0$ which are $SO(7)$, $G_2$ or $SL(2)/\pm 1$. Since all these groups are their own normalizers in $SO(7)$ in all cases we find that $G=G^0$. We now only have to exclude the cases $G=SO(7)$ and $G=SL(2)/\pm 1$. First suppose that $G=SL(2)/\pm 1$. \\
The group $SL(2)/\pm 1\cong SO(3)$ admits a faithful 3-dimensional representation 
\[\rho:SO(3)\rightarrow \GL(V).\] Let $\rho(\EE_i)$ be the connection associated to the representation 
\[\pidiff(\Gm,1)\rightarrow SO(3)\rightarrow \GL(V). \]
The connection $\rho(\EE_i)$ is a $3$-dimensional irreducible connection with slopes $\le 1/2$ at $\infty$ and which is regular singular at $0$. We have $\irr(\rho(\EE_i))\le 3/2$ and so either $\irr_\infty(\rho(\EE_i))=0$ or $\irr_\infty(\rho(\EE_i)))=1$. In the first case we have 
\[\rig(\rho(\EE_i))=\dim (\END(\rho(\EE_i))^{I_0})+\dim (\END(\rho(\EE_i))^{I_\infty})\ge 6\]
 which is a contradiction (recall that for any irreducible connection $\EE$ on some open subset $U$ of $\PP^1$ we always have $\rig(\EE)\le 2$). \\
In the second case, the formal type at $\infty$ of $\rho(\EE_i)$ has to be of the form 
\[\El(2,\alpha,1)\oplus (-1) \]
and we compute
\[\rig(\rho(\EE_i))=\dim \End(\rho(\EE_i))^{I_0}+2-1\ge 4 \]
which again yields a contradiction. \\
Now we're left with the cases $G=SO(7)$ and $G=G_2$. Recall that the third exterior power of the standard representation of $SO(7)$ is irreducible, so it suffices to prove that $G$ has a non-zero invariant in the third exterior power of its $7$-dimensional standard representation. This corresponds to the alternating Dickson trilinear form which is stabilized by $G_2$. In our case this amounts to finding horizontal sections of $\Lambda^3\EE_i$ for $i=1,2,3$, i.e. we have to show that $H^0(\Gm, \Lambda^3\EE_i)\neq 0$ or equivalently by duality that $H^2_c(\Gm, \Lambda^3 \EE_i)\neq 0$.
For this it suffices to prove that 
\[\chi(\PP^1,j_{!*}\Lambda^3\EE_i)>0.\]
Recall that 
\[\chi(\PP^1,j_{!*}\Lambda^3\EE_i)=\dim(\Lambda^3\EE_i)^{I_0}+\dim(\Lambda^3\EE_i)^{I_\infty}-\irr_\infty(\Lambda^3\EE_i)\]
as $0$ is a regular singularity. These invariants can be computed using Sabbah's formula for the determinant of elementary connections in \cite[Proposition 2.9]{Sa08}. For $i=1$, we have 
\begin{align*} &\, \Lambda^3(\El(2,\alpha_1,\lambda)\oplus\El(2,\alpha_1,\lambda^{-1})\oplus(\El(2,2\alpha_1,1)\oplus(-1)) \\
&=(\El(2,\alpha_1,\lambda)\otimes \det\El(2,\alpha_1,\lambda^{-1})) \oplus (\det\El(2,\alpha_1,\lambda)\otimes\El(2,\alpha_1,\lambda^{-1})) \\
&\oplus (\det\El(2,\alpha_1,\lambda^{-1})\oplus(\El(2,\alpha_1,\lambda)\otimes\El(2,\alpha_1,\lambda^{-1}))\oplus\det\El(2,\alpha_1,\lambda)) \\
&\otimes ((-1)\oplus\El(2,2\alpha_1,1)) \\
&\oplus(\El(2,\alpha_1,\lambda^{-1})\oplus\El(2,\alpha_1,\lambda))\otimes ((\El(2,2\alpha_1,1)\otimes (-1))\oplus \det(\El(2,2\alpha_1,1)) \\
&\oplus(\det\El(2,2\alpha_1,1)\otimes (-1))
\end{align*}
As the slopes in our case are of the form $1/p$ with $p>1$ all occuring determinant connections are regular. Therefore the irregularity of this connection is $13$. 
Since \[\det\El(2,2\alpha_1,1)\otimes (-1) \cong (-1)\otimes (-1) \cong (1)\] by \cite[Proposition 2.9]{Sa08} we also have $\dim(\Lambda^3\EE_1)^{I_\infty}\ge 1$. Finally we find that 
\[\chi(\PP^1,j_{!*}\Lambda^3\EE_1)=13+\dim(\Lambda^3\EE_1)^{I_\infty}-13\ge 1.\]
The second and thirds cases are completely analogous and we have 
\[\chi(\PP^1,j_{!*}\Lambda^3\EE_2)=13+4-15=2\]
and
\[\chi(\PP^1,j_{!*}\Lambda^3\EE_3)=9+\dim(\Lambda^3\EE_3)^{I_\infty}-\irr_\infty(\Lambda^3\EE_3)\ge 9+2-10=1. \]
Therefore for $i=1,2$ we have $\Gdiff(\EE_i)= G_2$. \par
For the systems with formal type $\El(6,\alpha_1, 1)\oplus(-1)$ at $\infty$ note that the systems in question have Euler characteristic $-1$ on $\Gm$ and therefore are hypergeometric by \cite[Theorem 3.7.1]{Ka90}. By \cite[4.1.]{Ka90} all these systems have differential Galois group $G_2$. \\
\textbf{The above list exhausts all cases.} Let $\EE$ be an irreducible irregular rigid $G_2$-connection, i.e. at some singularity the irregularity of $\EE$ is positive. By the rough classification of Section \ref{G2conn}, the only possibilities for $R(\EE)$ are 
\begin{align*}
&(0, 7, 7, 2), \\
&(0, 14, 13, 3), \\
&(0, 19, 17, 4) \ \textup{or} \\
&(0, 21, 19, 4). 
\end{align*}
Applying the same techniques as before, the only formal types left are those appearing in the above list together with one additional formal type which is given by the following table (here $\varepsilon$ denotes a primitive third root of unity).	
\begin{center}
\begin{tabular}{ c c }
	$0$ & $\infty$ \\
	\hline \\
	$(\varepsilon E_3,\varepsilon^{-1} E_3,1)$ & \specialcell[c]{$\El(2,\alpha_1,1) \oplus \El(2,\alpha_2,1)$ \\ $\oplus\,\El(2,\alpha_1+\alpha_2,1) \oplus (-1)$}
\end{tabular}
\end{center}
\medskip
The connection 
\[\EE:=\F((\varepsilon,\varepsilon^{-1})\otimes \phi^*(\F((\varepsilon^{-1},1,1,1,\varepsilon)\otimes\MC_{\varepsilon^{-1}}(\LL_5)))) \]
constructed from the rank one sheaf $\LL_5:=\LL_{(-\varepsilon,-1,-1,-1,\varepsilon^{-1})}$ on \[\PP^1-\{0,\frac{1}{4}\alpha_1^2,\frac{1}{4}\alpha_2^2,\frac{1}{4}(\alpha_1+\alpha_2)^2,\infty \}\]
has the above formal type. We will prove by contradiction that $\Gdiff(\EE)$ is not contained in $G_2$. Therefore suppose the contrary, i.e. $\Gdiff(\EE)\subset G_2$. As we have seen before, the morphism
\[\pidiff(\Gm,1)\rightarrow \GL_7(\IC)\]
corresponding to $\EE$ factors through $G_2(\IC)$. Denote by $\Ad$ the adjoint representation $\Ad:G_2\rightarrow \mathfrak{g}_2$. 
As $\EE$ is rigid and irreducible by construction, we find that
\[H^1(\PP^1,j_{!*}\Ad(\EE))=0\]
by \cite[Section 7]{Fr09}. We therefore have
\[0=\dim H^1(\PP^1,j_{!*}\Ad(\EE))=\irr_\infty(\Ad(\EE))-\dim \Ad(\EE)^{I_\infty}-\dim \Ad(\EE)^{I_0} \]
and the same for the connection $\EE_2$ we have constructed above. As the formal type at $\infty$ of $\EE$ and $\EE_2$ coincides, we find that 
\[\irr_\infty(\Ad(\EE))-\dim \Ad(\EE)^{I_\infty}=\irr_\infty(\Ad(\EE_2))-\dim \Ad(\EE_2)^{I_\infty}\]
and in particular a necessary condition for both connections to have differential Galois group $G_2$ is 
\[\dim\Ad(\EE)^{I_0}=\dim \Ad(\EE_2)^{I_0}.\]
These invariants are precisely the centraliser dimension of the local monodromy at $0$ of the connections in question. By Table 1 in \cite{Dett10}, $\dim \Ad(\EE_2)^{I_0}=6$ and $\dim\Ad(\EE)^{I_0}=8$ which yields a contradiction. Hence $\Gdiff(\EE)$ is not contained in $G_2$, concluding the proof. 

\begin{rem}

Let $\EE_{4,5}$ be the final system in the theorem with $x=\zeta_8$ a primitive $8$-th root of unity and $y=\zeta_8^2$ and denote by $[q]:\Gm\rightarrow \Gm$ the morphism given by $z\mapsto z^q$. In this setting we find that 
\[\EE_3\cong [2]^*\EE_{4,5}. \]
To see this we compute the pullback of the formal types. At the regular singularity, the pullback of the connection with monodromy $(\zeta_8,\zeta_8^2,\zeta_8^3,\zeta_8^5,\zeta_8^6,\zeta_8^7,1)$ has monodromy $(iE_2,-iE_2,-E_2,1)$. The pullback of $\El(6,\alpha_1, 1)\oplus(-1)$ is given due to \cite[2.5 \& 2.6]{Sa08} as
\[\El(3,\alpha,1)\oplus\El(3,\zeta_6^5\alpha,1)\oplus (1)\cong \El(3,\alpha,1)\oplus\El(3,-\alpha,1)\oplus (1),\]
since $\zeta_6^5\alpha=-\zeta_3^2\alpha$ and we can multiply by $\zeta_3$ to get $-\alpha$. By rigidity we get the desired isomorphism $\EE_3\cong [2]^*\EE_{4,5}$. \par
A similar analysis shows that systems in the second family $\EE_2$ with formal type \[\El(2,-\alpha_1,1)\oplus\El(2,\zeta_6^5\alpha_1,1)\oplus\El(2,\zeta_6^4,1)\oplus(-1)\]
at $\infty$ are pullbacks of the system $\EE_{4,4}$, the second to last system of the theorem, with $x=\zeta_3$ under the map $[3]:\Gm\rightarrow\Gm$. Of course not every system in the family $\EE_2$ is of this form and if they are not, they cannot be pullbacks of hypergeometrics (these would have to appear in the above list). 
\end{rem}
\bibliographystyle{jk}
\bibliography{mybib}

\end{document}